\documentclass[12pt]{amsart}

\usepackage{amsmath}
\usepackage{amssymb}
\usepackage{pifont}
\usepackage{epsfig}

\usepackage{amscd}
\usepackage{latexsym}
\usepackage{amsfonts}

\usepackage{framed}

\newtheorem{thm}{Theorem}[section]
\newtheorem{lem}[thm]{Lemma}
\newtheorem{cor}[thm]{Corollary}

\newtheorem{defi}[thm]{Definition}

\usepackage{graphicx}
\ifx\pdfoutput\undefined
\else
\fi
\graphicspath{{images/}}

\usepackage{color}

\newcommand\spt{{\rm spt}}
\newcommand\dive{{\rm div}}
\newcommand\tang{{\rm Tan}}
\newcommand\dist{{\rm dist}}

\numberwithin{equation}{section}

\begin{document}

\title[A fixed contact angle condition for varifolds]
{ A fixed contact angle condition \\ for varifolds}

\author{Takashi Kagaya}
\address{Department of Mathematics, Tokyo Institute of Technology,
152-8551, Tokyo, Japan}
\email{kagaya.t.aa@m.titech.ac.jp}

\author{Yoshihiro Tonegawa}
\address{Department of Mathematics, Tokyo Institute of Technology,
152-8551, Tokyo, Japan}
\email{tonegawa@math.titech.ac.jp}

\medskip



\thanks{T. Kagaya is partially supported by JSPS Research Fellow Grant number 16J00547 and Y. Tonegawa is partially supported by JSPS KAKENHI Grant Numbers (A) 25247008 and (S) 26220702. } 

\begin{abstract}
We define a generalized fixed contact angle condition for $n$-varifold and establish a boundary monotonicity formula.
The results are natural generalizations of those for the Neumann boundary condition 
considered by Gr\"{u}ter-Jost \cite{GJ}. 
\end{abstract}

\maketitle \setlength{\baselineskip}{18pt}

\section{Introduction}\label{sec:int}

Almgren proposed varifold as a generalized manifold in \cite{AF} and Allard 
established a number of fundamental properties of varifold such as rectifiability, 
compactness and regularity theorems in \cite{A1}. One of the key tools to 
analyze the local properties of varifold is the 
monotonicity formula \cite[Section 5]{A1}, which gives a good control of measure
whenever the first variation is well-behaved. 

When a varifold has a ``boundary'' in a suitable sense, one expects to have 
a modified monotonicity formula under a suitable set of assumptions. 
Loosely speaking, with prescribed $C^{1,1}$ boundary (which one may regard as ``Dirichlet 
boundary condtion''), Allard obtained a monotonicity formula at the boundary \cite{A} as well as
the regularity theorem up to the boundary. The result is improved recently by 
Bourni to $C^{1,\alpha}$ boundary \cite{Bourni}. For Neumann boundary condition,
which corresponds roughly to the prescribed right angle condition,
Gr\"{u}ter and Jost \cite{GJ} derived a monotonicity formula
by using a reflection technique and obtained the regularity theorem up to the boundary.
 
As a further inquiry, it 
is natural to extend the Neumann boundary condition to more general fixed contact angle 
condition. The condition arises naturally in various capillarity and
free boundary problems, where the boundary of domain under
consideration has non-trivial amount of surface energy. 
For such problem, Taylor \cite{Ta} established the
boundary regularity for area minimizing surfaces. More recently, De Philippis and Maggi \cite{PM} proved the
boundary regularity for minimizers of anisotropic surface 
energy. For minimizing problems, one does not necessarily need a 
monotonicity type formula since one can obtain upper and lower energy density
ratio bounds by some
energy comparison argument. Here, motivated by the dynamical problem such as the 
mean curvature flow,
we would like to investigate the notion of fixed contact angle condition for general varifolds
which do not necessarily correspond to  energy minimizing case. As far as we know, 
this aspect has not been studied in a general setting of varifold so far. 

In this paper, we introduce a notion of contact angle condition for general varifolds.  
Our condition is satisfied for smooth hypersurface having a fixed contact angle with the boundary of domain
under consideration. It is stated in terms of the first variation of varifolds and generalized mean curvature
vectors, and such condition is satisfied for a limit of diffused interface problem (\cite{KT2}). With a natural integrability condition on the generalized mean curvature vector, we prove that a 
modified monotonicity formula holds. 
The results are natural generalizations of those for the Neumann boundary condition 
considered by Gr\"{u}ter-Jost \cite{GJ}. 

The paper is organized as follows. 
Section \ref{sec:not} lists notation and recalls some well-known results from geometric measure theory. 
In Section \ref{sec:as}, we state the definition of fixed contact angle condition and 
discuss the implications such as the monotonicity formula. 
In Section \ref{sec:proof}, we prove the monotonicity formula
and we give a few final remarks in Section \ref{adrem}. 

\medskip

\section{Notation and basic definitions}\label{sec:not}

\subsection{Basic notation}

In this paper, $n$ will be a positive integer. 
For $0<r<\infty$ and $a \in \mathbb{R}^{n+1}$ let 
\[B_r(a) := \{ x \in \mathbb{R}^{n+1} : |x-a| < r\}. \]  
We denote by $\mathcal{L}^k$ the Lebesgue measure on $\mathbb{R}^k$ and by $\mathcal{H}^k$ the 
$k$-dimensional Hausdorff measure on $\mathbb{R}^{n+1}$ for each positive integer $k$. 
The restriction of $\mathcal{H}^k$ to a set $A$ is denoted by $\mathcal{H}^k \lfloor_A$.  
We let 
\[ \omega_k := \mathcal{L}^k(\{x \in \mathbb{R}^k : |x| < 1\}). \]

For any Radon measure $\mu$ on $\mathbb{R}^{n+1}$, $\phi \in C_c(\mathbb{R}^{n+1})$ and $\mu$ measurable set $A$, we often write 
\[ \mu(\phi) := \int_{\mathbb{R}^{n+1}} \phi \; d\mu, \quad \mu(A) := \int_{A} d\mu. \] 
Let the support of $\mu$ be 
\[ \spt(\mu) := \{x \in \mathbb{R}^{n+1} : \mu(B_r(x)) > 0 \; \mbox{for all} \; r>0 \}. \]
Let $\Theta^k(\mu, x)$ be the $k$-dimensional density of $\mu$ at $x$, i.e., 
\[ \Theta^k(\mu, x) := \lim_{r \to 0+} \dfrac{\mu(B_r(x))}{\omega_k r^k}, \]
if the limit exists. 

\subsection{Homogeneous maps and varifolds}

Let $\mathbf{G}(n+1,n)$ be the space of $n$-dimensional subspaces of $\mathbb{R}^{n+1}$. 
For $S \in \mathbf{G}(n+1,n)$, we identify $S$ with the corresponding orthogonal projection of $\mathbb{R}^{n+1}$
onto $S$. 
Let $S^\perp$ be the orthogonal complement of $S$ and we sometimes treat $S^\perp$ as the orthogonal projection $\mathbb{R}^{n+1}$ onto $S^\perp$. 
For two elements $A$ and $B$ of $\mbox{Hom}(\mathbb{R}^{n+1}; \mathbb{R}^{n+1})$, we define a scalar product as 
\[ A \cdot B := \sum_{i,j} A_{ij}B_{ij}. \]
The identity of $\mbox{Hom}(\mathbb{R}^{n+1}; \mathbb{R}^{n+1})$ is denoted by $I$. 
For these elements $A$ and $B$, we define the product, the operator norm and the spectrum norm as 
\[ (A \circ B)_{ij} := \sum_{k} A_{ik}B_{kj}, \quad \|A\| := \sup_{|x|=1} |Ax|, \quad |A| := \sqrt{A \cdot A }, \]
respectively. 

We recall some notions related to varifold and refer to \cite{A1,S} for more details. 
In what follows, let $X \subset \mathbb{R}^{n+1}$ be open and $G_n(X) := X \times \mathbf{G}(n+1, n)$. 
A general $n$-varifold in $X$ is a Radon measure on $G_n(X)$ and $\mathbf{V}_n(X)$ denotes the set of all general $n$-varifolds in $X$. 
For $V \in \mathbf{V}_n(X)$, let $\|V\|$ be the weight measure of V, namely, 
\[ \|V\|(\phi) := \int_{G_n(X)} \phi(x) \; dV(x,S) \quad \mbox{for} \; \phi \in C_c(X). \]
For any $\mathcal{H}^n$ measurable countably $n$-rectifiable set $M \subset X$ with locally finite $\mathcal{H}^n$ measure, there is a natural $n$-varifold 
$|M| \in \mathbf{V}_n(X)$ defined by 
\[ |M|(\phi) := \int_M \phi(x,\tang_xM) \; d\mathcal{H}^n(x) \quad \mbox{for} \; \phi \in C_c(G_n(X)), \]
where $\mbox{Tan}_xM \in \mathbf{G}(n+1, n)$ is the approximate tangent space which exists $\mathcal{H}^n$ a.e. on $M$. 
In this case, the weight measure of $|M|$ equals to $\mathcal{H}^n \lfloor_M$. 
We note that $n$-dimensional density of this varifold is equal to $1$ $\mathcal{H}^n$ a.e. on $M$.  

For $V \in \mathbf{V}_n(X)$, let $\delta V$ be the first variation of $V$, namely, 
\[ \delta V(g) := \int_{G_n(X)} \nabla g(x) \cdot S \; dV(x,S) \quad \mbox{for} \; g \in C^1_c(X;\mathbb{R}^{n+1}). \]
Let $\|\delta V\|$ be the total variation when it exists, and if $\delta V$ is absolutely continuous with respect to $\|V\|$, 
we have $\|V\|$ measurable $h$ with 
\[ \delta V(g) = - \int_X h \cdot g \; d\|V\| \quad \mbox{for} \; g \in C^1_c(X ; \mathbb{R}^{n+1}). \]
The vector field $h$ is called the generalized mean curvature vector of $V$. 

\section{Main results}\label{sec:as}

We first give a definition and then explain why it may be regarded as a generalized 
fixed contact angle condition for a varifold.  
\subsection{Fixed angle condition}
We assume that $\Omega\subset \mathbb R^{n+1}$ is a bounded open set with $C^2$ boundary
$\partial \Omega$. 
\begin{defi}
Given $V\in \mathbf V_n(\Omega)$ with $\|V\|(\Omega)<\infty$, 
$\mathcal H^n$ measurable set $B^+\subset\partial
\Omega$ and $\theta\in [0,\pi]$, we say that ``$V$ has a fixed contact angle $\theta$ with
$\partial \Omega$ at the boundary of $B^+$'' if the following conditions hold. 
\begin{itemize}
\item[(A1)] The generalized mean curvature vector $h$ exists, i.e., 
\begin{equation*}
\int_{G_n(\Omega)} \nabla g(x) \cdot S \; dV(x,S) = - \int_\Omega g(x) \cdot h(x) \; d\|V\|(x) \quad \mbox{for} \; g \in C^1_c(\Omega;\mathbb{R}^{n+1}). 
\end{equation*}
\item[(A2)] By setting $\sigma:=\cos\theta$, we have
\begin{equation}\label{firstvari} 
\int_{G_n(\Omega)} \nabla g(x) \cdot S \; dV(x,S) + \sigma \int_{B^+} \dive_{\partial \Omega} \; g(x) \; d\mathcal{H}^n(x) = - \int_{\Omega} g(x) \cdot h(x) \; d\|V\|(x) 
\end{equation}
for all $g \in C^1(\overline{\Omega}\,;\, \mathbb{R}^{n+1})$ with $g \cdot \nu_{\partial\Omega} = 0$ on $\partial\Omega$, where $\dive_{\partial\Omega}$ is the divergence on $\partial\Omega$ and 
$\nu_{\partial\Omega}$ is the outward unit normal vector on the boundary $\partial\Omega$. 
\end{itemize}
\label{fbc}
\end{defi}
Let us give a justification for the definition. 
Due to $\|V\|(\Omega)<\infty$, by setting $V=0$ outside of $\Omega$, we may extend $V$ to the entire $\mathbb R^{n+1}$ as an element in $\mathbf V_n(\mathbb R^{n+1})$ and we will regard $V$
in this way in the following. The condition (A1) means equivalently that the 
first variation $\delta V$ is absolutely continuous with respect to $\|V\|$. Geometrically
speaking, if $V=|M|$ with some smooth surface $M$, this means that there is no boundary
of $M$ in $\Omega$ since any presence of boundary in $\Omega$ gives a singular 
$\delta V$.  To better explain a motivation for the notion, let us assume the following boundedness
of the first variation in $\overline{\Omega}$, namely, 
\begin{equation}
\sup_{g\in C^1(\overline\Omega \,;\,\mathbb R^{n+1}),\, |g|\leq 1}
\delta V(g)<\infty.
\label{fv1}
\end{equation}
We may include \eqref{fv1} as ``(A3)'' in Definition \ref{fbc}, but as we will see, \eqref{fv1} is not needed to 
prove the subsequent monotonicity formula. 
With the zero extension of $V$, 
\eqref{fv1} means that $V$ has a bounded first variation $\delta V$ 
on $\mathbb R^{n+1}$. We should emphasize that this is typically different from the first
variation as an element of $\mathbf V_n(\Omega)$. The condition (A1)
implies that $\delta V\lfloor_{\Omega}=-h\|V\|\lfloor_{\Omega}$. On the other hand, 
$\delta V\lfloor_{\partial \Omega}$ is singular with respect to $\|V\|$ whenever
it is nonzero, since $\|V\|\lfloor_{\partial \Omega}=0$. 
By the definition of the first variation, we have 
\begin{equation*}
\int_{G_n(\Omega)} \nabla g(x)\cdot S\, dV(x,S)=\int_{G_n(\mathbb R^{n+1})}
\nabla g(x)\cdot S\, dV(x,S)=\int_{\overline\Omega} g(x)\cdot\, d(\delta V)(x).
\end{equation*}
Since $\delta V\lfloor_{\Omega}=-h\|V\|\lfloor_{\Omega}$, the condtion (A2) implies that
\begin{equation}
\int_{\partial\Omega} g\cdot\, d\,(\delta V)
+\sigma\int_{B^+} {\rm div}_{\partial\Omega} \,g\, d\mathcal H^{n}=0
\label{degio1}
\end{equation}
for all $g\in C^1(\overline\Omega \, ;\, \mathbb R^{n+1})$ with $g\cdot \nu_{\partial 
\Omega}=0$. If $\sigma=\cos\theta \neq 0$ (or $\theta\neq \pi/2$), \eqref{degio1} implies that $B^+$ has a finite 
perimeter in $\partial\Omega$. By De Giorgi's theorem (see \cite[Theorem 5.16]{EG}), the second term may be 
expressed as 
\begin{equation}
\sigma\int_{\partial^* B^+} g\cdot {\mathbf n}_{B^+}\, d\mathcal H^{n-1}
\label{degio2}
\end{equation}
where 
$\partial^* B^+$ is the reduced boundary of $B^+$ which is countably $(n-1)$-rectifiable
and ${\mathbf n}_{B^+}$ is the outer pointing unit normal to $\partial ^* B^+$ which
exists $\mathcal H^{n-1}$ a.e.~ on $\partial ^* B^+$. Now, define ${\mathbf n}_V=
\frac{\delta V}{\|\delta V\|}$ on $\partial \Omega$ so that ${\mathbf n}_V \|\delta V\|
=\delta V$. Then \eqref{degio1} and \eqref{degio2} mean that we have
\begin{equation}
\int_{\partial \Omega} g\cdot {\mathbf n}_V\, d\|\delta V\|+\sigma\int_{\partial^* B^+}
g\cdot \mathbf n_{B^+}\, d\mathcal H^{n-1}=0.
\label{degio3}
\end{equation}
Since $\mathbf n_{B^+}$ is tangent to $\partial\Omega$, 
\eqref{degio3} shows that
\begin{equation}
\label{degio4}
(\mathbf n_V -(\mathbf n_V\cdot \nu_{\partial \Omega})\nu_{\partial \Omega} ) 
\|\delta V\|=-\sigma \mathbf n_{B^+}\mathcal H^{n-1}\lfloor_{\partial^* B^+}
\end{equation}
on $\partial \Omega$. 

\begin{figure}[t]
\begin{center}
\includegraphics[width=6cm]{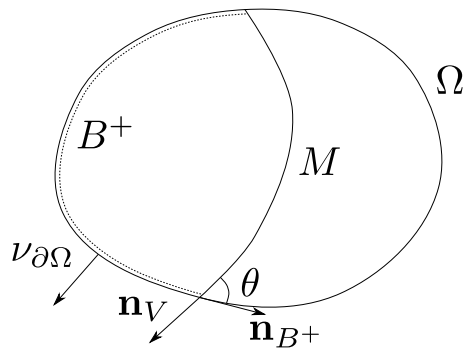}
\end{center}
\caption{The figure for the smooth surface $M$}
\label{fig1}
\end{figure}

Let us see what the above means in the case of smooth surfaces. 
Suppose that $V=|M|$ with a $C^2$ $n$-dimensional surface $M\subset \Omega$ without
boundary inside $\Omega$ but with a nontrivial boundary $\partial M$ in $\partial \Omega$. Suppose also that $B^+$ has a smooth boundary $\partial B^+$. 
Then $\mathbf n_V$ corresponds to the
unit co-normal of $\partial M$ pointing outwards. We also have $\|\delta V\|\lfloor_{\partial\Omega}=\mathcal H^{n-1}\lfloor_{\partial M}$ in this case. The reduced boundary
$\partial^* B^+$ is the usual boundary $\partial B^+$. 
Then \eqref{degio4} means that 
we need to have $\partial B^+\subset \partial M$,
and 
\begin{equation}
\begin{array}{ll} 
\mathbf n_V=\nu_{\partial\Omega}  & \mbox{ on }\partial M\setminus \partial B^+, \\
\mathbf n_V-(\mathbf n_V\cdot \nu_{\partial\Omega})\nu_{\partial\Omega}=
-\sigma\mathbf n_{B^+} & \mbox{ on }\partial B^+.
\end{array}
\label{degio5}
\end{equation}
The first case means that $\partial M$ intersects $\partial\Omega$ with 90 degree if it is not part of 
$\partial B^+$. The second case is precisely the fixed angle condition in the sense that
the angle formed by $\mathbf n_V$ and $-\mathbf n_{B^+}$ is $\theta$ (see figure \ref{fig1}). 
This is the reason for the definition of the fixed angle condition. 
Above discussion does not make sense unless \eqref{fv1} is satisfied, 
so we understand the definition in a weak sense in this case. Interestingly, for varifolds
arising from singular perturbation limit problems, finiteness of the first variation \eqref{fv1} 
is automatically 
satisfied, see \cite{KT2, MT}. On the other hand, the subsequent monotonicity
formula holds even without \eqref{fv1}, and we think it better if \eqref{fv1}
is not included as a part of the definition for a broader applicability of the notion. 

If we assume that there exists an open set $U\subset\mathbb R^{n+1}$ such that 
$M=\Omega\cap \partial U$ and $B^+=\partial \Omega\cap  U$, then, 
$\partial M= \partial B^+$ and the first case of \eqref{degio5} does not happen. 
In the application to the diffused interface limit problem with a fixed angle condition 
\cite{KT2}, the varifold $V$ arises as 
a  ``phase boundary'' in a sense and the presence of $B^+$ follows naturally. 

\subsection{The monotonicity formula}

The following Theorem \ref{mono} is a natural generalization of monotonicity formula of Gr\"{u}ter-Jost \cite{GJ} to 
the fixed angle condition defined in the previous subsection. To present the statement, 
we need some more notation. 
Suppose $\Omega\subset\mathbb R^{n+1}$ is a bounded open set with $C^2$ boundary 
$\partial\Omega$. Define $\varkappa$ as 
\[ \varkappa := \| \mbox{principal curvature of } \partial\Omega\|_{L^{\infty}(\partial\Omega)}.\]
For $s>0$, define a subset $N_{s}$ of $\mathbb{R}^{n+1}$ by 
\[ N_{s} := \{ x \in \mathbb R^{n+1} : \dist(x, \partial\Omega) < s \}.\]
For any boundary point $b\in\partial\Omega$ let 
\[ \tau(b) := \tang_{b}\,\partial\Omega \quad \mbox{and} \quad \nu(b) := \tau(b)^{\perp}. \]
There exists a sufficiently small 
\[s_0 \in (0, \varkappa^{-1}]\] depending only on $\partial \Omega$ 
such that all points $x \in N_{s_0}$ have a unique point $\xi(x) \in \partial \Omega$ such that $\dist(x, \partial \Omega) = |x - \xi(x)|$. 
By using this $\xi(x)$, we define the reflection point $\tilde{x}$ of $x$ with respect to $\partial \Omega$ as $\tilde{x} := 2\xi(x) - x$ 
and the reflection ball $\tilde{B}_r(a)$ of $B_r(a)$ with respect to $\partial \Omega$ as 
\[\tilde{B}_r(a) := \{ x \in \mathbb R^{n+1} : |\tilde{x} - a| < r \}.\] 
In the case for $x \in \partial\Omega$, we note that $\xi(x) = \tilde{x} = x$. 
If $y \in \mathbb{R}^{n+1}$, we set \[i_x (y) := \tau(\xi(x)) y - \nu(\xi(x)) y.\] 

\begin{thm}\label{mono}
Given $V\in \mathbf V_n(\Omega)$ with $\|V\|(\Omega)<\infty$,
$\mathcal H^n$ measurable set $B^+\subset\partial\Omega$ and $\theta\in [0,\pi/2]$,
suppose that $V$ has a fixed contact angle $\theta$ with $\partial \Omega$ at the boundary
of $B^+$, as in Definition \ref{fbc}. Assume that for some $p>n$ and $\Gamma\geq 0$, we
have 
\begin{equation}
\Big(\dfrac{1}{\omega_n}\int_{N_{s_0}\cap\Omega} 2|h(x)|^p\, d\|V\|(x)\Big)^{\frac1p}\leq \Gamma.
\end{equation}
Then there exists a constant $C\geq 0$ depending only on $n$ such that  
for any $x \in N_{s_0/6}\cap \overline\Omega$, 
\begin{equation}
\begin{split}
&\Big\{\frac{\|V\|(B_\rho(x)) + \|V\|(\tilde{B}_\rho(x)) + 2\sigma\mathcal{H}^n\lfloor_{B^+}(B_\rho(x))
}{\omega_n\rho^n}\Big\}^{\frac1p}\Big(1+C\varkappa\rho\Big(1+\dfrac{1}{p-n}\Big)\Big) \\
&+ \dfrac{\Gamma \rho^{1-\frac{n}{p}}}{p-n}
\end{split}
\label{monoeq}
\end{equation}
is a non-decreasing function of $\rho$ in $(0, s_0/6)$. Here $\sigma=\cos\theta$. 
In the case that $\theta\in (\pi/2,\pi]$, 
the same claim holds with $\sigma$ and $B^+$ replaced by $-\sigma$ and $\partial\Omega
\setminus B^+$ in \eqref{monoeq}, respectively. 
\end{thm}
See Section \ref{adrem} for more discussion.

\medskip

\section{Proof of Theorem \ref{mono}}\label{sec:proof}

The proof of the monotonicity formula \eqref{monoeq} is similar to that of 
Gr\"{u}ter-Jost \cite{GJ} except that we use \eqref{firstvari} with $\sigma\neq 0$.
For completeness, we present the proof in this section. 

First, we need to estimate the derivatives of $\xi(x)$ and $i_x$, and we 
cite the following lemma (\cite[Lemma 2.2]{A}). 

\begin{lem}\label{lem: deriofxi}
The following statements hold. 
\begin{itemize}
\item[(i)] $\xi$ is continuously differentiable in $N_{s_0}$. 
\item[(ii)] For $x \in N_{s_0}$, $Q(x) := \nabla \xi(x) - \tau(\xi(x))$ is symmetric, 
\begin{equation}\label{qperp} 
Q(x) \circ \nu(\xi(x)) = 0 
\end{equation}
and 
\begin{equation}\label{qderi} 
\|Q(x)\| \le \varkappa \dist(x, \partial \Omega) (1- \varkappa \dist(x, \partial \Omega))^{-1}. 
\end{equation} 
\item[(iii)] For $b \in \partial \Omega$, 
\begin{equation}\label{tanderi} 
\|\nabla_{\partial \Omega} \nu(b)\| = \|\nabla_{\partial \Omega} \tau(b)\| \le \varkappa 
\end{equation}
holds.  
\end{itemize}
\end{lem}

In addition, we need the following.
\begin{lem}\label{lem: ball}
Assume $a \in N_{s_0}$ and $\rho>0$ satisfy $\dist(a,\partial \Omega) \le \rho$ and $B_{\rho}(a) \subset N_{s_0}$. 
Then for any point $x \in \tilde{B}_\rho(a)$  
\begin{equation}\label{dist} 
\dist(x,\partial \Omega) \le 2 \rho 
\end{equation} 
and
\begin{equation}\label{ball} 
\tilde{B}_\rho(a) \subset B_{5\rho}(a). 
\end{equation}
\end{lem}

\begin{proof}
For any $x \in \tilde{B}_\rho(a)$, we have by the assumption $\dist(a,\partial \Omega) \le \rho$
\begin{align*}
\dist(x,\partial \Omega) =& \dist(\tilde{x},\partial \Omega) \le |\tilde{x} - a| + \dist(a,\partial \Omega) < 2\rho,
\end{align*}
which shows \eqref{dist}. 
Using \eqref{dist} we have 
\[ |x-a| \le |2\xi(x) - x - a| + 2|\xi(x) - x| = |\tilde{x}-a| + 2\dist(x,\partial \Omega) < 5\rho \]
and hence \eqref{ball} holds. 
\end{proof}

\begin{proof}[Proof of Theorem \ref{mono}] 
First, for $\theta\in (\pi/2,\pi]$, using $\int_{\partial\Omega}{\rm div}_{\partial\Omega}\,
g(x)\, d\mathcal H^n (x)=0$ for $g$ tangent to $\partial\Omega$, we may replace
$\sigma$ by $-\sigma$ and $B^+$ by $\partial\Omega\setminus B^+$ in \eqref{firstvari}.
Thus in the following, we may assume without loss of generality that $\sigma\geq 0$. 

For any point $a \in N_{s_0/6}\cap\overline{\Omega}$, we choose for \eqref{firstvari} the test function to prove the monotonicity formula around $a$. 
Let $\gamma \in C^1(\mathbb{R})$ satisfy 
\begin{equation*}
\gamma(t) = 
\begin{cases}
1, & t \le \dfrac{\rho}{2}, \\
0, & t \ge \rho
\end{cases}
\end{equation*} 
and $\gamma^{\prime}(t) \le 0$ for any $t$, where a constant $\rho>0$ satisfies $0 < \rho \le s_0/6$. 
Thus the boundary point $\xi(x)$ is defined for $x \in B_\rho(a)$ from $\dist(x, \partial \Omega) \le s_0/3$. 
Let the vector field $g$ be 
\[ g(x) = \gamma(r) (x-a) + \gamma(\tilde{r}) (i_x(\tilde{x}-a)), \]
where $r = |x-a|$ and $\tilde{r} = |\tilde{x}-a|$. 
This vector field $g$ may satisfy the property $g \cdot \nu_{\partial \Omega} = 0$ on $\partial \Omega$ by the following argument. 
For $x \in \partial \Omega$, $x=\tilde{x}=\xi(x)$ yields $r=\tilde{r}$ and 
\[ i_x (\tilde{x}-a) = \tau(x)(x-a) - \nu(x)(x-a). \]
Thus, we have 
\[ g(x) = \gamma(r) \{\tau(x)(x-a) + \nu(x)(x-a) \} + \gamma(r) \{\tau(x)(x-a) - \nu(x)(x-a)\} = 2\gamma(r) \tau(x)(x-a) \] 
and hence $g \cdot \nu_{\partial \Omega} = 0$ holds on $\partial \Omega$. 
To substitute $g$ in \eqref{firstvari}, we calculate the gradient of $g$ 
\begin{equation}\label{gradi} 
\nabla g(x) = \nabla (\gamma(r)(x-a)) + \nabla (\gamma(\tilde{r})(i_x(\tilde{x}-a))). 
\end{equation}
For the first term of right hand side of \eqref{gradi}, we have by a simple calculation 
\[ \nabla (\gamma(r)(x-a)) = r \gamma^{\prime}(r) \left(\dfrac{x-a}{r} \otimes \dfrac{x-a}{r}\right) + \gamma(r) I. \]
For the second term of right hand side of \eqref{gradi}, we calculate the following matrices $M$ and $N$: 
\begin{equation}\label{divide2}
\nabla (\gamma(\tilde{r})(i_x(\tilde{x}-a))) = \left(\dfrac{\partial}{\partial x_i}(\gamma(\tilde{r}))(i_x(\tilde{x}-a))_j\right) + \left( \gamma(\tilde{r}) \dfrac{\partial (i_x(\tilde{x}-a))_j}{\partial x_i} \right) =: (M_{ij}) + (N_{ij}). 
\end{equation}

Calculation of $M$: By the definitions $\tilde{x} = 2\xi(x)-x$ and $\tilde{r} = |\tilde{x}-a|$, we have 
\begin{align*}
\dfrac{\partial}{\partial x_i}(\gamma(\tilde{r})) =& \, \gamma^{\prime}(\tilde{r}) \dfrac{\partial}{\partial x_i} \sqrt{(2\xi_1(x) - x_1 - a_1)^2 + \cdots + (2\xi_{n+1}(x) - x_{n+1} - a_{n+1})^2} \\
=& \, \dfrac{\gamma^{\prime}(\tilde{r})}{\tilde{r}} \left\{\left(\sum_{k} 2\dfrac{\partial \xi_k(x)}{\partial x_i}(\tilde{x}-a)_k\right) - (\tilde{x}-a)_i\right\}. 
\end{align*}
Thus, the matrix element of $M$ is represented by 
\[ M_{ij} = \dfrac{\gamma^{\prime}(\tilde{r})}{\tilde{r}} \left\{\left(\sum_{k} 2\dfrac{\partial \xi_k(x)}{\partial x_i}(\tilde{x}-a)_k(i_x(\tilde{x}-a))_j\right) - (\tilde{x}-a)_i(i_x(\tilde{x}-a))_j\right\} \]
and hence, by $Q(x) = \nabla \xi(x) - \tau(\xi(x))$ and $i_x = \tau(\xi(x)) - \nu(\xi(x))$, we have 
\begin{align}
M =& \, \dfrac{\gamma^{\prime}(\tilde{r})}{\tilde{r}} \{2\nabla \xi(x) \circ ((\tilde{x}-a) \otimes i_x(\tilde{x}-a)) - (\tilde{x}-a) \otimes i_x(\tilde{x}-a) \} \label{calA}\\
=& \, \dfrac{\gamma^{\prime}(\tilde{r})}{\tilde{r}} \{2Q(x) \circ ((\tilde{x}-a) \otimes i_x(\tilde{x}-a)) + i_x \circ ((\tilde{x}-a) \otimes i_x(\tilde{x}-a)) \} \notag\\
=& \, \tilde{r}\gamma^{\prime}(\tilde{r}) \left\{ 2Q(x) \circ \left(\dfrac{\tilde{x}-a}{\tilde{r}} \otimes i_x\left(\dfrac{\tilde{x}-a}{\tilde{r}}\right) \right)+ i_x\left(\dfrac{\tilde{x}-a}{\tilde{r}}\right) \otimes i_x\left(\dfrac{\tilde{x}-a}{\tilde{r}}\right) \right\}. \notag
\end{align}

Calculation of $N$: We define the matrix $L(\xi(x))$ by $L(\xi(x)) = \tau(\xi(x)) - \nu(\xi(x)) = i_x$ and calculate as 
\begin{align*}
\dfrac{\partial}{\partial x_i}(i_x(\tilde{x}-a))_j =& \, \dfrac{\partial}{\partial x_i}\left(\sum_k L_{jk}(\xi(x))(\tilde{x}-a)_k\right) \\
=& \, \sum_k \left(\sum_l \dfrac{\partial L_{jk}}{\partial \xi_l}(\xi(x)) \dfrac{\partial \xi_l}{\partial x_i}(x)\right)(\tilde{x}-a)_k + \sum_k L_{jk}(\xi(x)) \left(2\dfrac{\partial \xi_k}{\partial x_i}(x) - \delta_{ik}\right), 
\end{align*}
where $\delta_{ik}$ is the Kronecker delta. 
By using this calculation, $Q(x) = \nabla \xi(x) - \tau(\xi(x))$ and \eqref{qperp}, we obtain 
\begin{align}
N =& \, \gamma(\tilde{r}) \{ \nabla \xi(x) \circ \nabla_{\partial \Omega} L(\xi(x)) \circ (\tilde{x}-a) + 2 \nabla \xi(x) \circ T(\xi(x)) - T(\xi(x)) \} \label{calB}\\
=& \, \gamma(\tilde{r}) \{ (Q(x) + \tau(\xi(x))) \circ \nabla_{\partial \Omega} (\tau - \nu) (\xi(x)) \circ (\tilde{x}-a) \notag\\
&+2 (Q(x) + \tau(\xi(x))) \circ (\tau(\xi(x)) - \nu(\xi(x))) - \tau(\xi(x)) + \nu(\xi(x)) \} \notag\\
=& \, \gamma(\tilde{r}) \{ (Q(x) + \tau(\xi(x))) \circ \nabla_{\partial \Omega} (\tau - \nu) (\xi(x)) \circ (\tilde{x}-a) + 2Q(x) + I\}, \notag
\end{align}
where $\cdot \circ \cdot \circ \cdot$ has to be interpreted appropriately. 

Substituting \eqref{calA} and \eqref{calB} in \eqref{divide2}, we have 
\begin{align*}
\nabla g(x) =& \, (\gamma(r)+\gamma(\tilde{r})) I + r\gamma^{\prime}(r)\left(\dfrac{x-a}{r} \otimes \dfrac{x-a}{r} \right) \\
& \, + \tilde{r}\gamma^{\prime}(\tilde{r}) \left\{i_x\left(\dfrac{\tilde{x}-a}{\tilde{r}} \right) \otimes i_x\left(\dfrac{\tilde{x}-a}{\tilde{r}} \right)\right\} 
+ \gamma(\tilde{r})\tilde{\varepsilon}_1(x) - \tilde{r}\gamma^{\prime}(\tilde{r}) \tilde{\varepsilon}_2(x), 
\end{align*}
where
\begin{align*}
&\tilde{\varepsilon}_1(x) = 2Q(x) + \{\tau(\xi(x)) + Q(x)\} \circ \nabla_{\partial \Omega}(\tau - \nu)(\xi(x)) \circ (\tilde{x} - a), \\
&\tilde{\varepsilon}_2(x) = -2Q(x) \circ \left(\dfrac{\tilde{x}-a}{\tilde{r}} \otimes i_x\left(\dfrac{\tilde{x}-a}{\tilde{r}}\right) \right). 
\end{align*}
Thus, for any $S \in \mathbf{G}(n+1,n)$ 
\begin{align*}
\nabla g(x) \cdot S =& \, n(\gamma(r) + \gamma(\tilde{r})) + r \gamma^{\prime}(r) (1-|S^{\perp}(\nabla r)|^2) \\
&\, + \tilde{r}\gamma^{\prime}(\tilde{r}) \left(1-\left|S^{\perp}\left(i_x\left(\dfrac{\tilde{x}-a}{\tilde{r}}\right)\right)\right|^2\right) 
+ \gamma(\tilde{r})\varepsilon_1(x,S) - \tilde{r}\gamma^{\prime}(\tilde{r})\varepsilon_2(x,S), 
\end{align*}
where  
\begin{align*}
&\varepsilon_1(x,S) = 2(S \cdot Q(x)) + \langle S \circ\{\tau(\xi(x)) + Q(x)\}, \nabla_{\partial \Omega}(\tau-\nu)(\xi(x)), \tilde{x}-a \rangle, \\
&\varepsilon_2(x,S) = -2\left\{Q(x)\left(\dfrac{\tilde{x}-a}{\tilde{r}}\right)\right\} \cdot \left\{S \circ i_x\left(\dfrac{\tilde{x}-a}{\tilde{r}}\right)\right\}
\end{align*}
and $\langle \cdot, \cdot, \cdot \rangle$ has to be interpreted appropriately. 
For $x \in \partial \Omega$, the properties $r = \tilde{r}$, $x=\tilde{x}$, $\dist(x,\partial \Omega)=0$ and \eqref{qderi} yield
\begin{align*}
\dive_{\partial \Omega} \; g(x) =& 2n\gamma(r) + r\gamma^{\prime}(r) \{(1-|\nu(x)(\nabla r)|^2) + (1-|\nu(x)(i_x(\nabla r))|^2)\} + \gamma(r)\varepsilon_3(x), 
\end{align*}
where 
\begin{align*}
\varepsilon_3(x) = \langle \tau(x), \nabla_{\partial \Omega}(\tau - \nu)(x), x-a \rangle. 
\end{align*}
Thus \eqref{firstvari}, $\sigma \ge 0$ and $\gamma^{\prime} \le 0$ imply 
\begin{align}\label{mono1}
&n\left(\int_\Omega \gamma(r) + \gamma(\tilde{r}) \; d\|V\| + 2\sigma \int_{B^+} \gamma(r) \; d\mathcal{H}^n \right) \\
&+ \int_\Omega r\gamma^{\prime}(r) + \tilde{r}\gamma^{\prime}(\tilde{r}) \; d\|V\| + 2\sigma \int_{B^+} r\gamma^{\prime}(r) \; d\mathcal{H}^n \notag\\ 
\le& -\int_\Omega \{\gamma(r)(x-a) + \gamma(\tilde{r})i_x(\tilde{x}-a)\} \cdot h \; d\|V\| \notag\\
&- \left(\int_{G_n(\Omega)} \gamma(\tilde{r})\varepsilon_1(x,S) - \tilde{r}\gamma^{\prime}(\tilde{r})\varepsilon_2(x,S) \; dV(x,S) \right) -\sigma \int_{B^+} \gamma(r)\varepsilon_3(x) \; d\mathcal{H}^{n}. \notag
\end{align}

Now take $\phi \in C^1(\mathbb{R})$ such that 
\[ \phi(t) = 1 \quad \mbox{for} \; t \le 1/2, \quad \phi(t) = 0 \quad \mbox{for} \; t \ge 1\]
and $\phi^{\prime}(t) \le 0$ for all $t$ and use \eqref{mono1} with $\gamma(r) = \phi(r/\rho)$. 
For this $\gamma$, 
\[ r\gamma^{\prime}(r) = -\rho\dfrac{\partial}{\partial \rho} \left\{\phi \left(\dfrac{r}{\rho}\right)\right\} \]
holds and implies 
\[ n I(\rho) - \rho I^{\prime}(\rho) \le -\{ H(\rho) +E_1(\rho) + \rho E_2^{\prime}(\rho) + E_3(\rho) \}, \]
where 
\begin{align*}
& I(\rho) = \int_\Omega \phi\left(\dfrac{r}{\rho}\right) + \phi\left(\dfrac{\tilde{r}}{\rho}\right) \; d\|V\| + 2\sigma \int_{B^+} \phi\left(\dfrac{r}{\rho}\right) \; d\mathcal{H}^n, \\
& H(\rho) = \int_\Omega \left\{ \phi\left(\dfrac{r}{\rho}\right)(x-a) + \phi\left(\dfrac{\tilde{r}}{\rho}\right)i_x(\tilde{x}-a) \right\} \cdot h \; d\|V\|, \\
& E_i(\rho) = \int_{G_n(\Omega)} \phi\left(\dfrac{\tilde{r}}{\rho}\right)\varepsilon_i(x,S) \; dV(x,S), \quad i=1,2, \\
& E_3(\rho) = \sigma \int_{B^+} \phi\left(\dfrac{r}{\rho}\right)\varepsilon_3(x) \; d\mathcal{H}^n. 
\end{align*}
Multiplying by $\rho^{-n-1}$ we have 
\begin{equation}\label{ine1}
\dfrac{d}{d\rho} \{ \rho^{-n}I(\rho) \} \ge \rho^{-n} E_2^{\prime}(\rho) + \rho^{-n-1}\{H(\rho) + E_1(\rho) + E_3(\rho)\}. 
\end{equation}
To estimate the right hand side of this inequality, we apply \eqref{qderi} and \eqref{dist}. 
>From $a \in N_{s_0/6} \cap \overline{\Omega}$ and $\rho \le s_0/6$, the distance $\dist (x, \partial \Omega) \le s_0/3 \le 1/(3\varkappa)$ for $x \in \tilde{B}_{\rho}(a)$ with $\rho \ge \dist(a,\partial \Omega)$. 
Hence $\varkappa \dist(x, \partial \Omega) \le 1/3$ and   
\begin{align*}
&|\varepsilon_i (x)| \le C\varkappa \rho, \quad \mbox{for} \; x \in \tilde{B}_\rho(a), \; i=1,2, \\
&|\varepsilon_3(x)| \le C\varkappa \rho \quad \mbox{for} \; x \in B^+ \cap B_\rho(a)  
\end{align*} 
with some constant $C$ depending only on $n$ whenever $\rho \ge \dist(a,\partial \Omega)$. 
These inequalities, $\sigma \ge 0$ and $\partial/(\partial \rho) \{\phi(\tilde{r}/\rho)\} \ge 0$ imply 
\begin{align*}
&|E_1(\rho)| \le C\varkappa \rho \int_\Omega \phi\left(\dfrac{\tilde{r}}{\rho}\right) \; d\|V\| \le C\varkappa\rho I(\rho), \\
&|E_2^{\prime}(\rho)| \le C\varkappa \rho \dfrac{\partial}{\partial \rho}\int_\Omega \phi\left(\dfrac{\tilde{r}}{\rho}\right) \; d\|V\| \le C\varkappa\rho I^{\prime}(\rho), \\
&|E_3 (\rho)| \le C\varkappa \rho \sigma \int_{B^+} \phi\left(\dfrac{r}{\rho}\right) d\mathcal{H}^n \le C\varkappa\rho I(\rho) 
\end{align*}
for any $\rho > 0$ because of $\tilde{B}_\rho(a) \cap \spt\|V\| = \emptyset$ and $B_\rho(a) \cap B^+ =\emptyset$ whenever $\rho < \dist(a,\partial \Omega)$. 
By applying the H\"{o}lder inequality and \eqref{ball}, we have 
\begin{align*} 
H(\rho) &\le \rho \left(\int_\Omega \phi\left(\dfrac{r}{\rho}\right) + \phi\left(\dfrac{\tilde{r}}{\rho}\right) \; d\|V\|\right)^{1-1/p}
\left(\int_\Omega \left(\phi\left(\dfrac{r}{\rho}\right) + \phi\left(\dfrac{\tilde{r}}{\rho}\right)\right)|h|^p \; d\|V\|\right)^{1/p} \\
&\le \rho I^{1-1/p}(\rho) \left(\int_{N_{s_0}\cap \Omega} 2 |h|^p \; d\|V\|\right)^{1/p} \le \rho \omega_n^{1/p} \Gamma I^{1-1/p}(\rho). 
\end{align*}
Combining these inequalities and \eqref{ine1}, we have 
\[ \dfrac{d}{d\rho} \{\rho^{-n} I(\rho)\} \ge -C\varkappa\rho^{-n}I(\rho) - C\varkappa\rho^{-n+1}I^{\prime}(\rho) - \omega_n^{1/p}\Gamma\rho^{-n}I^{1-1/p}(\rho) \]
and $p>n$ implies 
\begin{align*} 
&\, \dfrac{d}{d\rho} \{\rho^{-n} I(\rho)\}^{1/p} \\
\ge&\, -\dfrac{1}{p}\{C\varkappa\rho^{-n/p}I^{1/p}(\rho) + C\varkappa\rho^{1-n/p}I^{1/p-1}(\rho)I^{\prime}(\rho) + \omega_n^{1/p}\Gamma\rho^{-n/p} \} \\
\ge&\, -\dfrac{C\varkappa}{p}\left\{(1+p-n)\rho^{-n/p}I^{1/p}(\rho) + \left(1+ \dfrac{1}{p-n}\right)\rho^{1-n/p}I^{1/p-1}(\rho)I^{\prime}(\rho) \right\} - \dfrac{\omega_n^{1/p}\Gamma}{p}\rho^{-n/p}.
\end{align*}
Integrating from $\sigma$ to $s$ and integrating by parts, we may estimate 
\begin{align*}
&\{\rho^{-n} I(\rho)\}^{1/p} - \{s^{-n} I(s)\}^{1/p} \\
\ge& -C\varkappa\left(1+\dfrac{1}{p-n}\right)(\rho^{1-n/p}I^{1/p}(\rho) - s^{1-n/p}I^{1/p}(s)) - \omega_n^{1/p}\dfrac{\Gamma}{p-n}(\rho^{1-n/p}-s^{1-n/p}). 
\end{align*}
Rearranging terms and letting $\phi$ increase to $\chi_{(-\infty,1)}$ we have the monotonicity formula. 
\end{proof}

\section{Additional remarks}
\label{adrem}
As a consequence of the monotonicity formula, we may conclude the following. 
\begin{cor}
Under the same assumption of Theorem \ref{mono}, for any $x\in N_{s_0/6}\cap \overline\Omega$, 
\begin{equation}
\label{monolim}
\lim_{\rho\rightarrow 0+} \frac{\|V\|(B_\rho (x))+\|V\|(\tilde B_\rho (x))
+2\sigma\mathcal H^n\lfloor_{B^+}(B_\rho(x))}{\omega_n \rho^n}
\end{equation}
exists and it is upper-semicontinuous function of $x$. 
\end{cor}
For $x\in \partial\Omega$, $B_\rho(x)$ and $\tilde B_{\rho}(x)$
approach to each other as $\rho\rightarrow 0+$, thus if $\sigma=0$, 
this implies the existence of 
$\Theta^n(\|V\|,x)$ 
for $x\in \partial\Omega$. 
For $\sigma\neq 0$, 
since the existence of the third term of \eqref{monolim} is not guaranteed (only
up to $\mathcal H^n$ a.e.), 
we cannot conclude such existence in general on $\partial\Omega$.  If we assume,
in addition to (A1) and (A2), that \eqref{fv1} is satisfied, then $B^+$ has a finite
perimeter as discussed in the definition. 
Then by De Giorgi's theorem (cf. \cite[Theorem 5.19]{EG} used for characteristic function) again, for $\mathcal H^{n-1}$ a.e.
$x\in \partial\Omega$, the limit of the third term exists and is equal to 
either $2\sigma$, $\sigma$ or 0. In particular, on $\partial^* B^+$,
it is $\sigma$ for $\mathcal H^{n-1}$ a.e. 
Thus, in this case, we have $\Theta^n(\|V\|,x)$ 
for $\mathcal H^{n-1}$ a.e. on $\partial\Omega$ instead. 

It is also interesting to pursue a boundary regularity theorem under a natural
``closeness to a single sheet'' assumption. Extrapolating from \cite{A,GJ}, one may
assume that $\|V\|(B_\rho(x))$ is close to $\omega_n\rho^n/2$. On the other hand, one
needs to differentiate two cases, the case away from ``$\partial B^+$'' where the right
angle condition should be satisfied, and the other case of near ``$\partial B^+$'' 
where the fixed angle condition of $\theta$ should be satisfied. Compared to
\cite{GJ}, it is less clear what should be the right assumption for the further 
regularity theorem. 


\end{document}